\documentclass[preprint,12pt]{elsarticle}

\usepackage{amssymb}
\usepackage{amsmath}

\numberwithin{equation}{section}
\newtheorem{theorem}{Theorem}[section]
\newtheorem{corollary}[theorem]{Corollary}

\newtheorem{lemma}[theorem]{Lemma}
\newtheorem{proposition}[theorem]{Proposition}
\newtheorem{remark}[theorem]{Remark}
\newtheorem{example}[theorem]{Example}

\newtheorem{definition}[theorem]{Definition}

\newproof{proof}{Proof}

\journal{Comm. Algebra}

\begin{document}

\begin{frontmatter}

\title{A new class of generalized inverses in semigroups and rings with involution}

\author[1]{Huihui Zhu}
\ead{hhzhu@hfut.edu.cn}

\author[1]{Liyun Wu}
\ead{wlymath@163.com}

\author[2]{Jianlong Chen}
\ead{jlchen@seu.edu.cn}

\address[1]{School of Mathematics, Hefei University of Technology, Hefei 230009, China.}
\address[2]{School of Mathematics, Southeast University, Nanjing 210096, China.}

\begin{abstract} Let $S$ be a $*$-semigroup and let $a,w,v\in S$. The initial goal of this work is to introduce two new classes of generalized inverses, called the $w$-core inverse and the dual $v$-core inverse in $S$. An element $a\in S$ is $w$-core invertible if there exists some $x\in S$ such that $awx^2=x$, $xawa=a$ and $(awx)^*=awx$. Such an $x$ is called a $w$-core inverse of $a$. It is shown that the core inverse and the pseudo core inverse can be characterized in terms of the $w$-core inverse. Several characterizations of the $w$-core inverse of $a$ are derived, and the expression is given by the inverse of $w$ along $a$ and $\{1,3\}$-inverses of $a$ in $S$. Also, the connections between the $w$-core inverse and other generalized inverses are given. In particular, when $S$ is a $*$-ring, the existence criterion for the $w$-core inverse is given by units. The dual $v$-core inverse of $a$ is defined by the existence of $y\in S$ satisfying $y^2va=y$, $avay=a$ and $(yva)^*=yva$. Dual results for the dual $v$-core inverse also hold.
\end{abstract}


\end{frontmatter}



\section{Introduction}

Let $S$ be a semigroup. Following Drazin, an element $a\in S$ is Drazin invertible \cite{Drazin1958} if there exists some $x\in S$ such that
\begin{eqnarray*}
&&{\rm (i)}~ ax=xa,\\
&&{\rm (ii)}~ xax=x,\\
&&{\rm (iii)}~ a^k = a^{k+1}x ~ for~ some~ nonnegative~ integer ~k.
\end{eqnarray*}
Such an $x$ is called a Drazin inverse of $a$. It is unique if it exists, and is denoted by $a^D$. The smallest nonnegative integer $k$ in the condition (iii) is called the Drazin index of $a$, and is denoted by ${\rm ind}(a)$. The element $a$ is called group invertible if ${\rm ind}(a)=1$, and the group inverse of $a$ is denoted by $a^\#$. We denote by $S^D$ and $S^\#$ the sets of all Drazin invertible and group invertible elements in $S$, respectively. It is known that $a\in S^D$ if and only if $a^n\in a^{n+1}S \cap Sa^{n+1}$ for some positive integer $n$, and that $a\in S^\#$ if and only if $a\in a^2S \cap Sa^2$. In particular, if $a=a^2x=ya^2$ for some $x,y\in S$, then $a^\#=yax=y^2a=ax^2$.

Given a semigroup $S$, $S^1$ denotes the monoid generated by $S$. Following \cite{Green1951}, Green's preorders and relations are defined by

(i) $a\leq_\mathcal{L}b \Leftrightarrow S^1a \subset S^1b \Leftrightarrow$ there exists $x\in S^1$ such that $a=xb$.

(ii) $a\leq_\mathcal{R}b \Leftrightarrow aS^1 \subset bS^1 \Leftrightarrow$  there exists $y\in S^1$ such that $a=by$.

(iii) $a\leq_\mathcal{H}b \Leftrightarrow a\leq_\mathcal{L}b ~~{\rm and}~~ a\leq_\mathcal{R}b$.

(iv) $a\mathcal{L}b\Leftrightarrow S^1a = S^1b \Leftrightarrow$ there exist $x,y\in S^1$ such that $a=xb$ and $b=ya$.

(v) $a\mathcal{R}b\Leftrightarrow aS^1 = bS^1 \Leftrightarrow$ there exist $x,y\in S^1$ such that $a=bx$ and $b=ay$.

(vi) $a\mathcal{H}b \Leftrightarrow a\mathcal{L}b ~~{\rm and}~~ a\mathcal{R}b$.

Based on Green's preorders, Mary introduced the notion of the inverse along an element \cite{Mary2011}. Given any $a,d\in S$, the element $a$ is called invertible along $d$ if there exists some $b\in S$ such that $bad=d=dab$ and $b\leq_\mathcal{H}d$. Such an element $b$ is called the inverse of $a$ along $d$. It is unique if it exists, and is denoted by $a^{\parallel d}$. By $S^{\parallel d}$ we denote the set of all invertible elements along $d$ in $S$. The inverse along an element encompasses the classical invertibility, the group inverse and the Drazin inverse. Mary in \cite[Theorem 11]{Mary2011} illustrated that (i) $a$ is invertible if and only if $a^{\parallel 1}$ exists, (ii) $a\in S^\#$ if and only if $a^{\parallel a}$ exists, (iii) $a\in S^D$ if and only if $a^{\parallel a^n}$ exists for some positive integer $n$. In these cases, $a^{-1}=a^{\parallel 1}$, $a^\#=a^{\parallel a}$ and $a^D=a^{\parallel a^n}$. One also knows from \cite[Corollary 3.4]{Mary2013} that $a\in S^\#$ if and only if $1^{\parallel a}$ exists. Moreover, $1^{\parallel a}=aa^\#$.

In what follows, we assume that $S$ is a $*$-semigroup, that is a semigroup $S$ endowed with an involution $*$ satisfying $(x^*)^*=x$ and $(xy)^*=y^*x^*$ for all $x,y\in S$. 

Recall that an element $a\in S$ is said to have a  Moore-Penrose inverse \cite{Penrose1955} if there exists some $x\in S$ such that
\begin{center}
(1) $axa=a$, (2) $xax=x$, (3) $(ax)^*=ax$, (4) $(xa)^*=xa$.
\end{center}
Such an $x$ is called a Moore-Penrose inverse of $a$. It is unique if it exists, and is denoted by $a^\dag$. By $S^\dag$ we denote the set of all Moore-Penrose invertible elements in $S$. If $a,x\in S$ satisfy the equations $\{i_1,\ldots,i_k\}\subseteq \{1,2,3,4\}$, then $x$ is called a $\{i_1,\ldots,i_k\}$-inverse of $a$, and is denoted by $a^{(i_1,\ldots,i_k)}$. As usual, we denote by $S^{(1,3)}$ and $S^{(1,4)}$ the sets of all $\{1,3\}$-invertible and $\{1,4\}$-invertible elements in $S$, respectively. It is known that $a\in S^{(1,3)}$ if and only if $a\in Sa^*a$, and $a\in S^{(1,4)}$ if and only if $a\in aa^*S$. In particular, if $a=xa^*a$ for some $x\in S$, then $x^*$ is a $\{1,3\}$-inverse of $a$. If $a=aa^*y$ for some $y\in S$, then $y^*$ is a $\{1,4\}$-inverse of $a$. By $a\{1,3\}$ and $a\{1,4\}$ we denote the sets of all $\{1,3\}$-inverses and $\{1,4\}$-inverses of $a$, respectively. Mary \cite[Theorem 11]{Mary2011} also proved in a $*$-semigroup $S$ that $a\in S^\dag$ if and only if $a^{\parallel a^*}$ exists. Moreover, $a^\dagger=a^{\parallel a^*}$.

The core inverse and the dual-core inverse of complex matrices were firstly introduced by Baksalary and Trenkler in their paper \cite{Baksalary2010}. Suppose $A \in M_n(\mathbb{C})$, the ring of all $n$ by $n$ complex matrices. A matrix $X \in M_n(\mathbb{C})$ is called a
core inverse of $A$ if it satisfies $AX = P_A$ and $\mathcal{R}(X) \subseteq \mathcal{R}(A)$, where $\mathcal{R}(A)$
denotes the column space of $A$, and $P_A$ is the orthogonal projector onto $\mathcal{R}(A)$. Such a matrix $X$ is unique if it exists, and is denoted by $A^{\tiny{\textcircled{\#}}}$.  The dual-core inverse, when it exists, is the unique $A_{\tiny{\textcircled{\#}}}$ satisfying $A_{\tiny{\textcircled{\#}}}A = P_{A^*}$ and $\mathcal{R}(A_{\tiny{\textcircled{\#}}}) \subseteq \mathcal{R}(A^*)$.

Suppose that $R$ is a $*$-ring, that is an associative ring with an involution $*$ satisfying $(x^*)^*=x$, $(xy)^*=y^*x^*$ and $(x+y)^*=x^*+y^*$ for all $x,y\in R$. In 2014, Raki\'{c} et al. \cite{Rakic2014} extended the core inverse and the dual-core inverse of a complex matrix to the case of a $*$-ring $R$. It was proved that the core inverse of $a\in R$ is the solution of the following five equations
\begin{center}
 (1)~$axa=a$, (2)~$xax=x$, (3)~$ax^2=x$, (4)~$xa^2=a$, (5)~$(ax)^*=ax$.
\end{center}

Also, they \cite{Rakic2014} showed that the dual-core inverse of $a\in R$ is the solution of the following five equations
\begin{center}
$(1')$~$axa=a$, $(2')$~$xax=x$, $(3')$~$x^2a=x$, $(4')$~$a^2x=a$, $(5')$~$(xa)^*=xa$.
\end{center}

As usual, by $R^{\tiny{\textcircled{\#}}}$ and $R_{\tiny{\textcircled{\#}}}$ we denote the sets of all core invertible and dual-core invertible elements in $R$, respectively.

In 2017, Xu et al. \cite{Xu2017} found that the equations (1) and (2) above can be dropped, more precisely, they characterized the core inverse of $a\in R$ by the solution of the following three equations  \begin{center}
(3)~$ax^2=x$, (4)~$xa^2=a$, (5)~$(ax)^*=ax$.
\end{center}
The dual-core inverse can also be expressed by the solution of
 \begin{center}
 $(3')$~$x^2a=x$, $(4')$~$a^2x=a$, $(5')$~$(xa)^*=xa$.
\end{center}
In addition, they derived that (i) $a\in R^{\tiny{\textcircled{\#}}}$ if and only if $a\in R^\# \cap R^{(1,3)}$, and (ii) $a\in R_{\tiny{\textcircled{\#}}}$ if and only if $a\in R^\# \cap R^{(1,4)}$. Moreover, $a^{\tiny{\textcircled{\#}}}=a^\#aa^{(1,3)}$ and $a_{\tiny{\textcircled{\#}}}=a^{(1,4)}aa^\#$. Through the aspect of Mary's inverse along an element, the existence criteria of the core inverse and the dual-core inverse can be correspondingly stated: (i) $a\in R^{\tiny{\textcircled{\#}}}$ if and only if $1\in R^{\parallel a}$ and $a\in R^{(1,3)}$, and (ii) $a\in R_{\tiny{\textcircled{\#}}}$ if and only if $1\in R^{\parallel a}$ and $a\in R^{(1,4)}$. Moreover, $a^{\tiny{\textcircled{\#}}}=1^{\parallel a}a^{(1,3)}$ and $a_{\tiny{\textcircled{\#}}}=a^{(1,4)}1^{\parallel a}$. Another relation between the core inverse and the inverse along an element was discovered by Raki\'{c} et al. \cite[Theorem 4.3]{Rakic2014} who proved that $a$ is core invertible if and only if $a$ is invertible along $aa^*$ provided that $a\in R^\dag$, and that the two inverses coincide in this case.

In \cite{Gao2018}, Gao and Chen defined the pseudo core inverse (a.k.a. the core-EP inverse \cite{Manjunatha2013}) by three equations in $*$-rings. An element $a\in R$ is pseudo core invertible if there exists an $x\in R$ such that $xa^{m+1}=a^m$, $ax^2=x$ and $(ax)^*=ax$ for some positive integer $m$. Such an $x$ is called a pseudo core inverse of $a$. It is unique if it exists, and is denoted by $a^{\tiny{\textcircled{D}}}$.  The smallest positive integer $m$ is called the pseudo core index of $a$, and is denoted by ${\rm I}(a)$. We will use the symbol $R^{\tiny{\textcircled{D}}}$ to denote the set of all pseudo core invertible elements in $R$. One knows from \cite[Theorem 2.3]{Gao2018} that $a
\in R^{\tiny{\textcircled{D}}}$ with ${\rm I}(a)=m$ if and only if $a\in R^D$ with ${\rm ind}(a)=m$ and $a^k\in R^{(1,3)}$, for any integer $k\geq m$, and the relation $a^{\tiny{\textcircled{D}}}=a^Da^k(a^k)^{(1,3)}$ is also proved. More results on pseudo core inverses can be referred to \cite{Zhu2018}.

The paper is organized as follows. In Section 2, the $w$-core inverse and the dual $v$-core inverse are defined in a $*$-semigroup $S$. Then, several existence criteria are given. In particular, we show in Theorem \ref{relate to mary inverse} that $a$ is $w$-core invertible if and only if $w$ is invertible along $a$ and $a$ is $\{1,3\}$-invertible. Moreover, $w^{\parallel a}a^{(1,3)}$ is the $w$-core inverse of $a$. A dual result for the dual $v$-core inverse is also given in Theorem \ref{relate to dual mary inverse}. It is also shown that $a$ is $w$-core invertible if and only if $a^*$ is dual $w^*$-core invertible. In particular, we show that $a$ is pesudo core invertible if and only if $a^n$ is core invertible if and only if $a^n$ is $a$-core invertible for any integer $n\geq 1$. Also, $a$ is $a^*$-core invertible if and only if it is Moore-Penrose invertible if and only if it is dual $a^*$-core invertible. Finally, we show that the $w$-core inverse and the dual $v$-core inverse are both instances of Mary's inverses along an element and Drazin's $(b,c)$-inverses. In Section 3, all of our results are given by the language of ring theory. It is shown that $a$ is $w$-core invertible if and only if there exists a (unique) projection $p\in R$ such that $pa=0$ and $u=p+aw\in R^{-1}$. In Theorem \ref{vw intersect}, we characterize both $w$-core invertible and dual $v$-core invertible elements by units, under the assumption $v\in R^{\parallel a}$. Also, a counterexample is given to show that the assumption can not be removed in a general $*$-ring $R$. Specially, when $R$ is a Dedekind-finite ring, the hypothesis could be dropped (see Theorem \ref{vw intersect dedekind}). In Section 4, some applications of the $w$-core inverse are given in complex matrices.

\section{The $w$-core inverse in a $*$-semigroup}

In this section, we assume that $S$ is a $*$-monoid (a $*$-semigroup with unity 1). The goal in this section is to give several characterizations for the $w$-core inverse in $S$. For instance, Theorem \ref{relate to mary inverse} establishes the equivalence that $a$ is $w$-core invertible if and only if $w\in S^{\parallel a}$ and $a\in S^{(1,3)}$. Theorem \ref{relate to dual mary inverse} presents the equivalence that $a$ is dual $v$-core invertible if and only if $v\in S^{\parallel a}$ and $a\in S^{(1,4)}$.

\begin{definition}\label{ew}
Let $a,w\in S$. An element $a$ is called $w$-core invertible if there exists some $x\in S$ such that $awx^2=x$, $xawa=a$ and $awx=(awx)^*$. Such an $x$ is called a $w$-core inverse of $a$.
\end{definition}

\begin{lemma} \label{added lemma} For any $a,w\in S$, if $x\in S$ is a $w$-core inverse of $a$, then $awxa=a$ and $xawx=x$. Moreover, $wx$ is a $\{1,2,3\}$-inverse of $a$.
\end{lemma}

\begin{proof} As $x$ is a $w$-core inverse of $a$, then $awx^2=x$, $xawa=a$ and $(awx)^*=awx$. Hence, $a=xawa=(awx^2)awa=(awx)xawa=awxa$, and $x=awx^2=(xawa)wx^2=xaw(awx^2)=xawx$. So, $wx$ is a $\{1,2,3\}$-inverse of $a$.
\hfill$\Box$
\end{proof}

Given any $a,w\in S$, we prove in Theorem \ref{ideal form} below that $a$ is $w$-core invertible if and only if $a\in awS$ and $aw$ is core invertible. Moreover, $(aw)^{\tiny\textcircled{\tiny{\#}}}$ is the $w$-core inverse of $a$. Several results afterwards will come for free. Most fundamentally, we have

\begin{theorem}\label{uniqueness} Let $a,w\in S$. Then $a$ has at most one $w$-core inverse in $S$.
\end{theorem}

In view of Theorem \ref{uniqueness}, it is known that the $w$-core inverse is unique if it exists. The $w$-core inverse of $a$ is denoted by $a_w^{\tiny\textcircled{\tiny{\#}}}$. We denote by $S_w^{\tiny\textcircled{\tiny{\#}}}$ the set of all $w$-core invertible elements in $S$.

It is not difficult to observe that the $1$-core inverse is just the classical core inverse. So, core invertible elements are $w$-core invertible. However, $w$-core invertible elements may not be core invertible as the following example shows.

\begin{example} \label{Ex1} {\rm Let $S$ be the semigroup of all $2 \times 2$ complex matrices and let the involution $*$ be the conjugate transpose. Suppose $a=\begin{bmatrix}
0 & 1 \\
0 & 0 \\
\end{bmatrix}
$, $w=\begin{bmatrix}
           3 & 6 \\
           1 & 0 \\
         \end{bmatrix}
\in S$. Then $a$ is $w$-core invertible and
$a_w^{\tiny\textcircled{\tiny{\#}}}=
\begin{bmatrix}
1 & 0 \\
0 & 0 \\
\end{bmatrix}
$. Clearly, $a \notin S^\#$, and hence $a\notin S^{\tiny\textcircled{\tiny{\#}}}$.}
\end{example}

The following theorem, a main result of this paper, presents the representation of the $w$-core inverse of $a$ by the inverse of $w$ along $a$ and its $\{1,3\}$-inverses, where $a,w\in S$. First, an auxiliary lemma about the existence criterion of the inverse along an element is given.

\begin{lemma} \label{Mary left and right} {\rm \cite[Theorem 2.2]{Mary2013}} Let $a,d\in S$. Then $a\in S^{\parallel d}$ if and only if $d\leq_\mathcal{H} dad$. In this case, $a^{\parallel d}=dx=yd$, where $x,y\in S$ satisfy $d=dadx=ydad$.
\end{lemma}

\begin{theorem}\label{relate to mary inverse} Let $a,w\in S$. Then $a\in S_{w}^{\tiny\textcircled{\tiny{\#}}}$ if and only if $w^{\parallel a}$ and $a^{(1,3)}$ both exist. In this case, $a_{w}^{\tiny\textcircled{\tiny{\#}}}=w^{\parallel a}a^{(1,3)}$ and $w^{\parallel a}=a_w^{\tiny\textcircled{\tiny{\#}}}a$.
\end{theorem}

\begin{proof} First suppose $x\in S$ is the $w$-core inverse of $a$. Then $a=xawa\in Sawa$ and $a=awxa=aw(awx^2)a\in awaS$, which give $a\in awaS \cap Sawa$, and hence $w\in S^{\parallel a}$ by Lemma \ref{Mary left and right}.

Again, by Lemma \ref{added lemma}, we have $(awx)^*=awx$ and $awxa=a$, and hence $a\in S^{(1,3)}$.

Conversely, if $w^{\parallel a}$ and $a^{(1,3)}$ exist, then $x=w^{\parallel a}a^{(1,3)}$ is the $w$-core inverse of $a$. Indeed, we have

(1) Since $w^{\parallel a}\in aS$, there exists some $y\in S$ such that $w^{\parallel a}=ay$ and hence $aa^{(1,3)}w^{\parallel a}=aa^{(1,3)}ay=ay=w^{\parallel a}$. So, $awx^2=(aww^{\parallel a})a^{(1,3)}w^{\parallel a}a^{(1,3)}=aa^{(1,3)}w^{\parallel a}a^{(1,3)}=w^{\parallel a}a^{(1,3)}=x$.

(2) Note that $w^{\parallel a}\in Sa$. Then $w^{\parallel a}a^{(1,3)}a=w^{\parallel a}$ and $xawa=w^{\parallel a}a^{(1,3)}awa=w^{\parallel a}wa=a$.

(3) $(awx)^*=awx$ since $awx=aww^{\parallel a}a^{(1,3)}=aa^{(1,3)}$.
\hfill$\Box$
\end{proof}

\begin{remark}{\rm In Theorem  \ref{relate to mary inverse} above, the $w$-core inverse of $a\in S$ is expressed by the product of $w^{\parallel a}$ and $a^{(1,3)}$. It is well known that an element could have different $\{1,3\}$-inverses. However, the product of $w^{\parallel a}$ and $a^{(1,3)}$, i.e., $w^{\parallel a}a^{(1,3)}$ is unique, that is for $x,y\in a\{1,3\}$, we have $w^{\parallel a}x=w^{\parallel a}y$. Indeed, the equality $ax=ayax=(ay)^*(ax)^*=(axay)^*=(ay)^*=ay$ implies $w^{\parallel a}x=w^{\parallel a}y$ since $w^{\parallel a}\in Sa$.}
\end{remark}

Applying Lemma \ref{group result} below, we can obtain another representation of the $w$-core inverse.

\begin{lemma} \label{group result} {\rm \cite[Theorem 7]{Mary2011}} Let $a,w\in S$. Then the following conditions are equivalent{\rm:}

\emph{(i)} $w\in S^{\parallel a}$.

\emph{(ii)} $aw \mathcal{R} a$ and $aw\in S^\#$.

\emph{(iii)} $wa \mathcal{L} a$ and $wa\in S^\#$.

In this case, $w^{\parallel a}=a(wa)^\#=(aw)^\# a$.
\end{lemma}

\begin{corollary} \label{extended repre} Let $a,w\in S$. Then $a\in S_{w}^{\tiny\textcircled{\tiny{\#}}}$ if and only if $w^{\parallel a}$ and $a^{(1,3)}$ both exist. In this case, we have $a_{w}^{\tiny\textcircled{\tiny{\#}}}=a(wa)^\#a^{(1,3)}=(aw)^\#aa^{(1,3)}$.
\end{corollary}

By Definition \ref{ew} above, it is clear that if $a\in S_{w}^{\tiny\textcircled{\tiny{\#}}}$ then $aw\in S^{\tiny\textcircled{\tiny{\#}}}$ for any $a,w\in S$. However, the converse statement does not hold in general. The following result presents under what conditions the converse statement holds.

\begin{theorem}\label{ideal form} Let $a,w\in S$ and let $n\geq 2$ be an integer. Then the following conditions are equivalent{\rm:}

\emph{(i)} $a\in S_w^{\tiny\textcircled{\tiny{\#}}}$.

\emph{(ii)} $a\in S[(aw)^*]^{n}a\cap S(aw)^{n-1}a$.

\emph{(iii)} $a\in awS$ and $aw\in S^{\tiny\textcircled{\tiny{\#}}}$.

In this case, $a_w^{\tiny\textcircled{\tiny{\#}}}=(aw)^{\tiny\textcircled{\tiny{\#}}}=w^{\parallel a}w(aw)^{(1,3)}$.
\end{theorem}

\begin{proof}

(i) $\Rightarrow$ (ii) As $a$ is $w$-core invertible, then there exists some $x\in S$ such that $xawa=a$, $awx^2=x$ and $(awx)^*=awx$, which guarantee $xawaw=aw$, $awx^2=x$ and $(awx)^*=awx$. So $aw\in S^{\tiny\textcircled{\tiny{\#}}}$. In terms of \cite[Theorem 2.10]{Li2018} (although this result was given in a $*$-ring, it does hold in a $*$-semigroup), $aw\in S^{\tiny\textcircled{\tiny{\#}}}$ if and only if $aw\in S[(aw)^*]^naw \cap S(aw)^n$ for all integers $n\geq 2$. It follows from Theorem \ref{relate to mary inverse} that $a\in S_{w}^{\tiny\textcircled{\tiny{\#}}}$ implies $w\in S^{\parallel a}$ and hence $a\in awaS\subseteq awS$. So, $a\in S[(aw)^*]^{n}a\cap S(aw)^{n-1}a$.

(ii) $\Rightarrow$ (iii) Given $a\in S[(aw)^*]^{n}a\cap S(aw)^{n-1}a$, then $aw\in S[(aw)^*]^{n}aw\cap S(aw)^n$ and so $aw\in S^{\tiny\textcircled{\tiny{\#}}}$. Since $a\in S[(aw)^*]^{n}a$, there exists some $x\in S$ such that $a=x[(aw)^*]^{n}a=x[(aw)^*]^{n-1}w^*a^*a \subseteq Sa^*a$. So, $a\in S^{(1,3)}$ and $w(aw)^{n-1}x^*$ is a \{1,3\}-inverse of $a$. We have at once $a=aw(aw)^{n-1}x^*a\in awS$.

(iii) $\Rightarrow$ (i) Let $x\in S$ be the core inverse of $aw$. Then $awx^2=x$, $awx=(awx)^*$ and $x(aw)^2=aw$. Since $a\in awS$, there is some $t\in S$ such that $a=awt=x(aw)^2t=xawa$. Therefore, $a\in S_w^{\tiny\textcircled{\tiny{\#}}}$ and $a_w^{\tiny\textcircled{\tiny{\#}}}=(aw)^{\tiny\textcircled{\tiny{\#}}}$.

Note that if $aw\in S^{\tiny\textcircled{\tiny{\#}}}$ then $(aw)^{\tiny\textcircled{\tiny{\#}}}=(aw)^\#aw(aw)^{(1,3)}$. Consequently, $(aw)^{\tiny\textcircled{\tiny{\#}}}=w^{\parallel a}w(aw)^{(1,3)}$ by Lemma \ref{group result}.  So, $a_w^{\tiny\textcircled{\tiny{\#}}}=(aw)^{\tiny\textcircled{\tiny{\#}}}=w^{\parallel a}w(aw)^{(1,3)}$. \hfill$\Box$
\end{proof}

\begin{remark} \label{rmk1} {\rm In Theorem \ref{ideal form}, (i) $\Leftrightarrow$ (ii) does not hold for $n=1$, i.e., $a\in S(aw)^*a\cap Sa$ can not imply $a\in S_w^{\tiny\textcircled{\tiny{\#}}}$. Such as, let $S$ be an infinite complex matrix semigroup whose rows and columns are both finite and let the involution $*$ be the conjugate transpose. Suppose $a=\sum_{i=1}^{\infty}e_{i+1,i}$ and $w=1$. Then $a^*a=1$ and $a\in S(aw)^*a\cap Sa$. However, $w^{\parallel a}=1^{\parallel a}$ does not exist, so that $a$ is not $w$-core invertible by Theorem \ref{relate to mary inverse}. In fact, if $1^{\parallel a}$ exists, then, by Lemma \ref{Mary left and right}, $a\in a^2S \cap S a^2$, and consequently $a=a^2s$ for some $s\in S$. So, $1=a^*a=a^*a^2s=as$, which together with $a^*a=1$ to guarantee that $a$ is invertible. A contradiction.}
\end{remark}

Given any $a,w_1, w_2 \in S$ with $w_1 \neq w_2$, if $a_{w_1}^{\tiny\textcircled{\tiny{\#}}}$ and $a_{w_2}^{\tiny\textcircled{\tiny{\#}}}$ both exist, then $a_{w_1}^{\tiny\textcircled{\tiny{\#}}}$ and $a_{w_2}^{\tiny\textcircled{\tiny{\#}}}$ are not equal in general. There are, of course, lots of examples to illustrate this fact. However, we find an interesting counterexample such that $a_{w_1}^{\tiny\textcircled{\tiny{\#}}}=a_{w_2}^{\tiny\textcircled{\tiny{\#}}}= \cdots =a_{w_n}^{\tiny\textcircled{\tiny{\#}}}$, for different $w_i$ ($i=1,2,\cdots, n$). Such as, let $S$ and the involution $*$ be the same as that of the previous Example \ref{Ex1}. Take $a=
\begin{bmatrix}
0 & 1 \\
0 & 0 \\
\end{bmatrix}\in S$, then $a_{w_i}^{\tiny\textcircled{\tiny{\#}}}=
\begin{bmatrix}
1 & 0 \\
0 & 0 \\
\end{bmatrix}$ for any $w_i$ of the form
$\begin{bmatrix}
* & * \\
1 & 0 \\
\end{bmatrix}$.

In order to extend \cite [Theorem 2.2]{Drazin2021} from rings $R$ to semigroups $S$, Drazin interpreted right (left) annihilators in a general semigroup $S$. Following Drazin, given any $a\in S$, the right annihilator of $a$ is defined by $a^0=\{(r,s)\in S^1 \times S^1:ar=as\}$, and the left annihilator of $a$ is defined by ${^0}a=\{(p,q)\in S^1\times S^1:pa=qa\}$. If $S$ is a ring, the right annihilator of $a$ is usually defined by $a^0=\{x\in R:ax=0\}$ and the left annihilator of $a$ is usually defined by ${^0}a=\{x \in R:xa =0\}$. When $S$ is a ring, we claim the fact that $(r,s)\in a^0$ if and only if $r-s\in a^0$.

We next present the relations between Green's preorders and Drazin's left (right) annihilators in $S$.

\begin{lemma} \label{Green Drazin}  Let $a,b\in S$. Then we have

\emph{(i)} If $a\leq_\mathcal{R}b$, then ${^0}b \subseteq {^0}a$.

\emph{(ii)} If $a\leq_\mathcal{L}b$, then $b^0 \subseteq a^0$.

\emph{(iii)} If $ a\mathcal{R}b$, then ${^0}a={^0}b$.

\emph{(iv)} If $a\mathcal{L}b$, then $a^0=b^0$.

\end{lemma}

\begin{proof} (i) If $a\leq_\mathcal{R}b$, i.e., $aS\subseteq bS$, then $a=bt$ for some $t\in S$. Suppose $(p,q)\in {^0}b$. Then $pb=qb$ and hence $pa=p(bt)=(pb)t=(qb)t=q(bt)=qa$, so that $(p,q)\in {^0}a$.

(ii) can be proved similarly.

(iii) and (iv) follow from (i) and (ii).
\hfill$\Box$
\end{proof}

\begin{theorem}\label{characteristic ew} Let $a,w\in S$. Then the following conditions are equivalent{\rm:}

\emph{(i)} $a$ is $w$-core invertible.

\emph{(ii)}  There exists some $x\in S$ such that $(1)$ $awxa=a$, $(2)$ $xawx=x$, $(3)$ $(awx)^*=awx$, $(4)$ $xawa=a$  and $(5)$ $awx^2=x$.

\emph{(iii)} There exists some $x\in S$ such that $awxa=a$, $xS=aS$  and  $Sx=Sa^*$.

\emph{(iv)} There exists some $x\in S$ such that $awxa=a$, $^0x={^0}a$  and  $x^0=(a^*)^0$.

\emph{(v)} There exists some $x\in S$ such that $awxa=a$, $^0x={^0}a$  and  $(a^*)^0 \subseteq x^0$.
\end{theorem}

\begin{proof}

(i) $\Rightarrow$ (ii) by Lemma \ref{added lemma}.

(ii) $\Rightarrow$ (iii) Given (ii), we have $(4)$ $xawa=a$  and $(5)$ $awx^2=x$, which imply $aS=xawaS\subseteq xS$ and $xS=awx^2S\subseteq aS$. So, $xS=aS$. Also, by $(2)$ $xawx=x$ and $(3)$ $(awx)^\ast=awx$, we have $Sx=Sxawx=Sx(awx)^*=Sxx^*w^*a^*\subseteq Sa^*$. Note that $(1)$ $awxa=a$ and $(3)$ $(awx)^*=awx$. Then $Sa^*=S(awxa)^*=Sa^*(awx)^*=Sa^*awx\subseteq Sx$. Therefore, $Sx=Sa^*$.

(iii) $\Rightarrow$ (iv) follows from Lemma \ref{Green Drazin}.

(iv) $\Rightarrow$ (v) is clear.

(v) $\Rightarrow$ (i) It follows from $awxa=a$ that $(1,awx)\in {^0}a$. Since $^0x={^0}a$, we have $(1,awx)\in {^0}x$, this gives $awx^2=x$. Note that $a^*=a^*(awx)^*$. Then $(1,(awx)^*)\in (a^*)^0 \subseteq x^0$, and hence $x=x(awx)^*$, so that $awx=awx(awx)^*=(awx)^*$. This in turn implies $x=xawx$ and $(1,xaw)\in {^0}x={^0}a$, we have $xawa=a$. So, $a$ is $w$-core invertible. \hfill$\Box$
\end{proof}

Set $w=1$ in Theorem \ref{characteristic ew}, we get the characterization for the core inverse, which extends some results of \cite{Rakic2014} from a $*$-ring to a $*$-semigroup.

\begin{corollary} \label{core char} Let $a\in S$. Then the following conditions are equivalent{\rm:}

\emph{(i)} $a$ is core invertible.

\emph{(ii)}  There exists some $x\in S$ such that $(1)$ $axa=a$, $(2)$ $xax=x$, $(3)$ $(ax)^*=ax$, $(4)$ $xa^2=a$  and $(5)$ $ax^2=x$.

\emph{(iii)} There exists some $x\in S$ such that $axa=a$, $xS=aS$  and  $Sx=Sa^*$.

\emph{(iv)} There exists some $x\in S$ such that $axa=a$, $^0x={^0}a$  and  $x^0=(a^*)^0$.

\emph{(v)} There exists some $x\in S$ such that $axa=a$, $^0x={^0}a$  and  $(a^*)^0 \subseteq x^0$.
\end{corollary}

We next show that the core inverse, the pseudo core inverse and the Moore-Penrose inverse are special cases of the $w$-core inverse. In Proposition \ref{core another} below, we show that $a\in S$ is core invertible if and only if it is $a$-core invertible. Also, it is proved in Proposition \ref{core another 1} that $a$ is pseudo core invertible with pseudo core index $n$ if and only if $a^n$ is $a$-core invertible if and only if $a^n$ is core invertible, where $n \geq 1$ is an integer. In a $*$-semigroup, we prove in Proposition \ref{* core another} that $a\in S$ is Moore-Penrose invertible if and only if it is $a^*$-core invertible if and only if it is dual $a^*$-core invertible.

\begin{proposition} \label{core another} Let $a\in S$. Then the following conditions are equivalent{\rm:}

\emph{(i)} $a\in S^{\tiny\textcircled{\tiny{\#}}}$.

\emph{(ii)} $a\in S^\# \cap S^{(1,3)}$.

\emph{(iii)} $a\in S_a^{\tiny\textcircled{\tiny{\#}}}$.

\emph{(iv)} There exists some $x\in S$ such that $a^2x^2=x$, $xa^3=a$ and $(a^2x)^*=a^2x$.

In this case, $a^{\tiny\textcircled{\tiny{\#}}}=aa_a^{\tiny\textcircled{\tiny{\#}}}$ and $a_a^{\tiny\textcircled{\tiny{\#}}}=a^\#a^{\tiny\textcircled{\tiny{\#}}}$.
\end{proposition}

\begin{proof}

(i) $\Leftrightarrow$ (ii) by taking $w=1$ in Theorem \ref{relate to mary inverse}.

(ii) $\Leftrightarrow$ (iii) by Theorem \ref{relate to mary inverse} and $a\in S^\# \Leftrightarrow a\in S^{\parallel a}$.

(iii) $\Leftrightarrow$ (iv) by taking $w=a$.

It is known that $x\in S$ satisfying the condition (iv) is the $a$-core inverse of $a$, and $x=a^{\parallel a}a^{(1,3)}=a^\#a^{(1,3)}$.

So, $a^{\tiny\textcircled{\tiny{\#}}}=ax=aa_a^{\tiny\textcircled{\tiny{\#}}}$ and $x=a_a^{\tiny\textcircled{\tiny{\#}}}=a^\#a^{\tiny\textcircled{\tiny{\#}}}$.
\hfill$\Box$
\end{proof}

Here is a consequence of Theorem \ref{ideal form} and Proposition \ref{core another}. Given any $a\in S$, then $a$ is core invertible if and only if $a$ is $a$-core invertible if and only if $a^2$ is core invertible and $a\in a^2S$ if and only if $a^2$ is $a^2$-core invertible and $a\in a^2S$ if and only if $a^4$ is core invertible and $a\in a^2S$ and $a^2\in a^4S$ if and only if $a^4$ is $a^4$-core invertible and $a\in a^2S$ and $a^2\in a^4S$ if and only if $a^8$ is core invertible and $a\in a^2S$ and $a^2\in a^4S$ and $a^4\in a^8S$. These equivalences can be written continually. Note the fact that $a\in a^2S$ implies $a^2\in a^4S$. Indeed, given $a\in a^2S$, then there is some $t\in S$ such that $a=a^2t=a(a^2t)t=a^3t^2=\cdots =a^nt^{n-1}\in a^nS$ for any integer $n\geq 1$, hence $a^2\in a^4S$ and $a^4\in a^8S$. So, we claim the fact that $a$ is $a$-core invertible if and only if $a$ is core invertible if and only if $a^p$ is core invertible and $a\in a^pS$ for some integer $p\geq 1$.

It is of interest to consider whether the equivalence above holds when the power of $a$ is no less than one. Precisely, whether $a^n$ is core invertible is equivalent to that $a^n$ is $a$-core invertible, for any integer $n\geq 1$. The following result gives a positive answer.

\begin{proposition} \label{core another 1} Let $a\in S$ and let $n\geq 1$ be an integer. Then the following conditions are equivalent{\rm:}

\emph{(i)} $a \in S^{\tiny{\textcircled{D}}}$ with {\rm I}$(a)=n$.

\emph{(ii)} $a^n \in S_a^{\tiny\textcircled{\tiny{\#}}}$.

\emph{(iii)} $a^n\in S^{\tiny\textcircled{\tiny{\#}}}$.

In this case, $a^{\tiny{\textcircled{D}}}={a^n}(a^n)_{a}^{\tiny\textcircled{\tiny{\#}}}=a^{n-1}(a^n)^{\tiny\textcircled{\tiny{\#}}}$.
\end{proposition}

\begin{proof} (i) $\Rightarrow$ (ii) Suppose $a \in S^{\tiny{\textcircled{D}}}$ with I$(a)=n$. Then, by \cite[Theorem 2.3]{Gao2018}, $a\in S^D$ with ind$(a)=n$ and $a^n\in S^{(1,3)}$. Again, it follows from \cite[Theorem 11]{Mary2011} (2) that $a\in S^D$ with ind$(a)=n$ yields $a\in S^{\parallel {a^n}}$. So, $a^n \in S_a^{\tiny\textcircled{\tiny{\#}}}$ by Theorem \ref{relate to mary inverse}.

(ii) $\Rightarrow$ (iii) Given $a^n \in S_a^{\tiny\textcircled{\tiny{\#}}}$, then, by Theorem \ref{relate to mary inverse}, $a\in S^{\parallel a^n}$ and $a^n\in S^{(1,3)}$. Since $a\in S^{\parallel a^n}$, we have at once $a^n\in a^{2n+1}S \cap Sa^{2n+1} \subseteq a^{2n}S \cap Sa^{2n}$, and consequently $a^n\in S^\#$. So, $a^n\in S^{\tiny\textcircled{\tiny{\#}}}$.

(iii) $\Rightarrow$ (i) As $a^n\in S^{\tiny\textcircled{\tiny{\#}}}$, then $a^n\in S^\# \cap S^{(1,3)}$. To show (i), it suffices to prove $a\in S^D$ with ind$(a)=n$ by \cite[Theorem 2.3]{Gao2018}. Once given $a^n\in S^\#$, then $a\in S^D$ with ind$(a)=n$ and $a^D=a^{n-1}(a^n)^\#$ (see, e.g., \cite[page 1111]{Zhu2019}), as required.

Let $x=a^{\parallel a^n}(a^n)^{(1,3)}$ be the $a$-core inverse of $a^n$. Then $x=a^D(a^n)^{(1,3)}$ since $a^D=a^{\parallel a^n}$ in terms of \cite[Theorem 11]{Mary2011} (2). So, by \cite[Theorem 2.3]{Gao2018}, $a^{\tiny{\textcircled{D}}}=a^Da^n(a^n)^{(1,3)}=a^n a^D(a^n)^{(1,3)}=a^nx$. Similarly, we have $a^{\tiny{\textcircled{D}}}=a^n a^D(a^n)^{(1,3)}=a^n(a^{n-1})(a^n)^\#(a^n)^{(1,3)}=a^{n-1}a^n(a^n)^\#(a^n)^{(1,3)}=a^{n-1}(a^n)^{\tiny\textcircled{\tiny{\#}}}$.
\hfill$\Box$
\end{proof}

Dually, we can give the definition of the dual $v$-core inverse in a $*$-semigroup $S$. For any $a,v\in S$, $a$ is called dual $v$-core invertible if there exists some $y\in S$ such that $avay=a$, $y^2va=y$ and $yva=(yva)^*$. Such an $y$ is called a dual $v$-core inverse of $a$. The dual $v$-core inverse of $a$ is unique if it exists, and is denoted by $a_{v,\tiny\textcircled{\tiny{\#}}}$. By $S_{v,\tiny\textcircled{\tiny{\#}}}$ we denote the set of all dual $v$-core invertible elements in $S$. Several characterizations of the dual $v$-core inverse are given below.

\begin{theorem}\label{characteristic vf} Let $a,v\in S$. Then the following conditions are equivalent{\rm:}

\emph{(i)} $a$ is dual $v$-core invertible.

\emph{(ii)} There exists some $y\in S$ such that $(1')$ $ayva=a$, $(2')$ $yvay=y$, $(3')$ $(yva)^\ast=yva$, $(4')$ $avay=a$  and $(5')$ $y^2va=y$.

\emph{(iii)} There exists some $y\in S$ such that $ayva=a$, $yS=a^*S$  and  $Sy=Sa$.

\emph{(iv)} There exists some $y\in S$ such that $ayva=a$, ${^0}y={^0}(a^*)$  and  $y^0=a^0$.

\emph{(v)} There exists some $y\in S$ such that $ayva=a$, ${^0}(a^*)\subseteq {^0}y$  and  $a^0=y^0$.
\end{theorem}

\begin{theorem} \label{relate to dual mary inverse} Let $a,v\in S$. Then $a\in S_{v,\tiny\textcircled{\tiny{\#}}}$ if and only if $v^{\parallel a}$ and $a^{(1,4)}$ both exist. In this case, $a_{v,\tiny\textcircled{\tiny{\#}}}$=$a^{(1,4)}v^{\parallel a}=a^{(1,4)}a(va)^\#=a^{(1,4)}(av)^\#a$.
\end{theorem}

Combining with Theorems \ref{relate to mary inverse} and \ref{relate to dual mary inverse}, we have the following result.

\begin{proposition} \label{wv core char} Let $a,w,v\in S$. Then $a\in S_w^{\tiny\textcircled{\tiny{\#}}} \cap S_{v,\tiny\textcircled{\tiny{\#}}}$ if and only if $w,v\in S^ {\parallel a}$ and $a\in S^{\dag}$.
\end{proposition}

For any $a,w\in S$, it is known from Lemma \ref{Mary left and right} that $w$ is invertible along $a$ if and only if $a\in awaS \cap Sawa$ if and only if $a^*\in a^*w^*a^*S \cap Sa^*w^*a^*$ if and only if $w^*$ is invertible along $a^*$. Moreover, $(w^*)^{\parallel {a^*}}=(w^{\parallel a})^*$. One also knows that $a\in S^{(1,3)}$ if and only if $a^*\in S^{(1,4)}$. Moreover, $(a^*)^{(1,4)}=(a^{(1,3)})^*$. We hence have the following result.

\begin{proposition} Let $a,w\in S$. Then $a$ is $w$-core invertible if and only if $a^*$ is dual $w^*$-core invertible. In this case, $(a_w^{\tiny\textcircled{\tiny{\#}}})^*=(a^*)_{w^*, \tiny\textcircled{\tiny{\#}}}$.
\end{proposition}

\begin{proof}
By Theorems \ref{relate to mary inverse} and \ref{relate to dual mary inverse}, we have $(a^*)_{w^*, \tiny\textcircled{\tiny{\#}}}=(a^*)^{(1,4)}(w^*)^{\parallel a^*}=(a^{(1,3)})^*(w^{\parallel a})^*=(w^{\parallel a}a^{(1,3)})^*=(a_w^{\tiny\textcircled{\tiny{\#}}})^*$.
\hfill$\Box$
\end{proof}

\begin{lemma} \label{MP inverse ideal char} {\rm \cite[Theorem 3.12]{Zhu2017}} Let $a\in S$. The following conditions are equivalent{\rm:}

\emph{(i)} $a\in S^{\dag}$.

\emph{(ii)} $a\in aa^*aS$.

\emph{(iii)} $a\in Saa^*a$.

In this case, $a^{\dag}=(ax)^*=(ya)^*$, where $x,y\in S$ satisfy $a=aa^*ax=yaa^*a$.
\end{lemma}

The following result shows that the existence of the $a^*$-core inverse of $a$ coincides with the existence of its dual $a^*$-core inverse, which are indeed equivalent to the existence of its Moore-Penrose inverse.

\begin{proposition} \label{* core another} Let $a\in S$. Then the following conditions are equivalent{\rm :}

\emph{(i)} $a\in S_{a^*}^{\tiny\textcircled{\tiny{\#}}}$.

\emph{(ii)} $a\in S^\dag$.

\emph{(iii)} $a\in S_{a^*,\tiny\textcircled{\tiny{\#}}}$.

In this case, $a^\dag=(a_{a^*}^{\tiny\textcircled{\tiny{\#}}}a)^*=(aa_{a^*,\tiny\textcircled{\tiny{\#}}})^*$, $a_{a^*}^{\tiny\textcircled{\tiny{\#}}}=(a^\dag)^*a^\dag$ and $a_{a^*,\tiny\textcircled{\tiny{\#}}}=a^\dag(a^\dag)^*$.
\end{proposition}

\begin{proof} (i) $\Rightarrow$ (ii) Suppose $x\in S$ is the $a^*$-core inverse of $a$. Then $a=xaa^*a\in Saa^*a$. Hence, $a\in S^\dag$ and $a^\dag=(xa)^*$ by Lemma \ref{MP inverse ideal char}.

(ii) $\Rightarrow$ (iii) To show $a\in S_{a^*,\tiny\textcircled{\tiny{\#}}}$, it suffices to prove $(a^*)^{\parallel a}$ (i.e., $a\in S^\dag$, and hence $a\in S^{(1,4)}$). The condition (ii) immediately gives this.

(iii) $\Rightarrow$ (i) From Theorems \ref{relate to mary inverse} and \ref{relate to dual mary inverse}, it is known that $a\in S_{a^*,\tiny\textcircled{\tiny{\#}}}$ if and only if $(a^*)^{\parallel a}$ exists if and only if $(a^*)^{\parallel a}$ and $a^{(1,3)}$ both exist if and only if $a\in S_{a^*}^{\tiny\textcircled{\tiny{\#}}}$.

By a direct calculation, $(a^\dag)^*a^\dag$ is the $a^*$-core inverse of $a$, and $a^\dag(a^\dag)^*$ is the dual $a^*$-core inverse of $a$.
\hfill$\Box$
\end{proof}

For the group inverse, the Drazin inverse and the core inverse, one knows that $(a^\#)^\#=a^2a^\#$, $(a^D)^D=a^2a^D$ and $(a^{\tiny\textcircled{\tiny{\#}}})^{\tiny\textcircled{\tiny{\#}}}=a^2a^{\tiny\textcircled{\tiny{\#}}}$.

It is natural to ask whether the $w$-core inverse also shares a similar property, i.e., whether $w$-core invertible elements are $w$-core invertible. In fact, the answer to this question is negative. See the following example.

\begin{example} {\rm Let $S$ and the involution $*$ be the same as that of the previous Example \ref{Ex1}. Take $a=
\begin{bmatrix}
0 & 1 \\
0 & 0 \\
\end{bmatrix}, w=\begin{bmatrix}
0 & 0 \\
1 & 0 \\
\end{bmatrix}\in S$, then $a_{w}^{\tiny\textcircled{\tiny{\#}}}=
\begin{bmatrix}
1 & 0 \\
0 & 0 \\
\end{bmatrix}$. However, $a_{w}^{\tiny\textcircled{\tiny{\#}}}\notin S_{w}^{\tiny\textcircled{\tiny{\#}}}$ as there exists no $x\in S$ such that $xa_{w}^{\tiny\textcircled{\tiny{\#}}}wa_{w}^{\tiny\textcircled{\tiny{\#}}}=a_{w}^{\tiny\textcircled{\tiny{\#}}}$ since $a_{w}^{\tiny\textcircled{\tiny{\#}}}wa_{w}^{\tiny\textcircled{\tiny{\#}}}=0$.}
\end{example}

It is of interest to study whether $w$-core invertible elements are core invertible in a $*$-semigroup. The following result provides a positive answer.

\begin{theorem} Let $a,w\in S$ and let $a\in S_w^{\tiny\textcircled{\tiny{\#}}}$. Then $a_w^{\tiny\textcircled{\tiny{\#}}}\in S^{\tiny\textcircled{\tiny{\#}}}$ and $(a_w^{\tiny\textcircled{\tiny{\#}}})^{\tiny\textcircled{\tiny{\#}}}=(aw)^2a_w^{\tiny\textcircled{\tiny{\#}}}$.
\end{theorem}

\begin{proof} It follows from Theorem \ref{ideal form} (i) $\Rightarrow$ (iii) that $aw\in S^{\tiny\textcircled{\tiny{\#}}}$ and $a_w^{\tiny\textcircled{\tiny{\#}}}=(aw)^{\tiny\textcircled{\tiny{\#}}}$. So, $(a_w^{\tiny\textcircled{\tiny{\#}}})^{\tiny\textcircled{\tiny{\#}}}=((aw)^{\tiny\textcircled{\tiny{\#}}})^{\tiny\textcircled{\tiny{\#}}}
=(aw)^2(aw)^{\tiny\textcircled{\tiny{\#}}}=(aw)^2a_w^{\tiny\textcircled{\tiny{\#}}}$.
\hfill$\Box$
\end{proof}

At the end of this section, we aim to show that our defined $w$-core inverses and dual $v$-core inverses are instances of two classes of outer generalized inverses.

\begin{theorem} \label{wv mary} Let $a,w,v\in S$ and $a\in S^\dag$. Then

\emph{(i)} $a\in S_w^{\tiny\textcircled{\tiny{\#}}}$ if and only if $aw$ is invertible along $aa^*$. In this case, the $w$-core inverse of $a$ coincides with the inverse of $aw$ along $aa^*$.

\emph{(ii)} $a\in S_{v,{\tiny\textcircled{\tiny{\#}}}}$ if and only if $va$ is invertible along $a^*a$. In this case, the dual $v$-core inverse of $a$ coincides with the inverse of $va$ along $a^*a$.
\end{theorem}

\begin{proof} (i) Suppose $a\in S_w^{\tiny\textcircled{\tiny{\#}}}$ and $x$ is the $w$-core inverse of $a$. Then, $awxa=a$, $xawx=x$, $xawa=a$, $awx^2=x$ and $awx=(awx)^*$ by Theorem \ref{characteristic ew}. We now show that $x=a_w^{\tiny\textcircled{\tiny{\#}}}$ is the inverse of $aw$ along $d=aa^*$. Indeed,

(1) $xawd=xawaa^*=aa^*=d=(aa^*)^*=(awxaa^*)^*=aa^*(awx)^*=aa^*awx=dawx$.

(2) $x=awx^2=a(a^\dag a)^*wx^2=aa^*(a^\dag)^*wx^2=d(a^\dag)^*wx^2 \in dS$.

(3) $x=xawx=x(awx)^*=x(wx)^*a^*=x(wx)^*a^\dag aa^*=x(wx)^*a^\dag d\in Sd$.

So, $x=a_w^{\tiny\textcircled{\tiny{\#}}}$ is the inverse of $aw$ along $aa^*$.

Conversely, let $x=(aw)^{\parallel aa^*}$ be the inverse of $aw$ along $aa^*$. Then $xawaa^*=aa^*=aa^*awx$ and $x\in aa^*S \cap Saa^*$. We next show that $x$ is the $w$-core inverse of $a$.

As $a\in S^\dag$, then $awx=(aa^\dag)^* awx=(a^\dag)^*a^*awx=(a^\dag aa^\dag)^*a^*awx=(a^\dag)^*a^\dag aa^*awx=(a^\dag)^*a^\dag aa^*=(a^\dag)^*a^*=aa^\dag$. This means $awx=(awx)^*$ and $awxa=a$.

Since $xawaa^*=aa^*$, we have $xawa=a$ by the implication $a\in S^\dag \Rightarrow a\in aa^*S$.

It follows from $x\in aa^*S$ that $x=aa^*s$ for some $s\in S$, and consequently $awx^2=awx(aa^*s)=(awxa)a^*s=aa^*s=x$.

So, $x$ is the $w$-core inverse of $a$.

(ii) can be proved similarly.
\hfill$\Box$
\end{proof}

In terms of Proposition \ref{core another} and Theorem \ref{wv mary}, we have the following corollary, among them, (i) and (ii) were essentially given in \cite[Theorem 4.3]{Rakic2014}.

\begin{corollary} \label{Rakic theorem 4.3} Let $a\in S^\dag$. Then the following conditions are equivalent{\rm :}

\emph{(i)} $a\in S^{\tiny\textcircled{\tiny{\#}}}$.

\emph{(ii)} $a$ is invertible along $aa^*$.

\emph{(iii)} $a\in S_a^{\tiny\textcircled{\tiny{\#}}}$.

\emph{(iv)} $a^2$ is invertible along $aa^*$.

In this case, $a^{\tiny\textcircled{\tiny{\#}}}=a^{\parallel aa^*}=aa_a^{\tiny\textcircled{\tiny{\#}}}=a(a^2)^{\parallel aa^*}$.
\end{corollary}

\begin{remark} \label{referee add} {\rm One knows from Theorem \ref{ideal form} (i) $\Rightarrow$ (iii) if $a\in S_w^{\tiny\textcircled{\tiny{\#}}}$ then $aw\in S^{\tiny\textcircled{\tiny{\#}}}$ and $a_w^{\tiny\textcircled{\tiny{\#}}}=(aw)^{\tiny\textcircled{\tiny{\#}}}$. It gives $(aw)\in S^{\parallel aw(aw)^*}$ and  $(aw)^{\tiny\textcircled{\tiny{\#}}}=(aw)^{\parallel aw(aw)^*}$ by Corollary \ref{Rakic theorem 4.3}, provided that $aw\in S^\dag$. Herein it follows from Theorem \ref{wv mary} that $a\in S_w^{\tiny\textcircled{\tiny{\#}}}$ gives $aw\in S^{\parallel aa^*}$ and $a_w^{\tiny\textcircled{\tiny{\#}}}=(aw)^{\parallel aa^*}$ provided that $a\in S^\dag$. Note that the two inverses of $aw$ along $aw(aw)^*$ and $aa^*$ coincide. We then claim that $aw(aw)^*$ and $aa^*$ belong to the same $\mathcal{H}$-class, and thus $aa^*S=aw(aw)^*S$. It follows from \cite{Mary2011} that if $a\in S^\dag$ then $aa^*$ is a trace product and $a\in aa^*S$. So, we also get $a\in aw(aw)^*S$.}
\end{remark}

Given any $a,b,c\in S$, an element $a$ is $(b,c)$-invertible \cite{Drazin2012} if there exists some $y\in S$ such that $yab=b$, $cay=c$ and $y\in bSy \cap ySc$, or equivalently, $yay=y$, $yS=bS$ and $Sy=Sc$. Such an $y$ is called a $(b,c)$-inverse of $a$.

The connections about $w$-core inverses and $(b,c)$-inverses can be given as follows.

\begin{theorem}  \label{relations b,c-inverses} Let $a,w,v\in S$. Then

\emph{(i)} $a\in S_w^{\tiny\textcircled{\tiny{\#}}}$ if and only if $aw$ is $(a,a^*)$-invertible. In this case, the $w$-core inverse of $a$ coincides with the $(a,a^*)$-inverse of $aw$.

\emph{(ii)} $a\in S_{v,\tiny\textcircled{\tiny{\#}}}$ if and only if $va$ is $(a^*,a)$-invertible. In this case, the dual $v$-core inverse of $a$ coincides with the $(a^*,a)$-inverse of $va$.
\end{theorem}

\begin{proof}

We only prove the statement (i) as the statement (ii) can be proved similarly.

Suppose that $a\in S_w^{\tiny\textcircled{\tiny{\#}}}$ and $x$ is the $w$-core inverse of $a$. Then, by Theorem \ref{characteristic ew}, we have $xawx=x$, $xS=aS$ and $Sx=Sa^*$. So, $x$ is the $(a,a^*)$-inverse of $aw$.

Conversely, suppose that $aw$ is $(a,a^*)$-invertible and $x$ is the $(a,a^*)$-inverse of $aw$. Then, $xawa=a$, $a^*awx=a^*$, $xS= aS$ and $Sx=Sa^*$.

Pre-multiplying $a^*awx=a^*$ by $(wx)^*$ gives $(awx)^*awx=(wx)^*a^*awx=(wx)^*a^*=(awx)^*$ and $awx=(awx)^*$, so that $a^*awx=a^* \Rightarrow a^*=a^*(awx)^*=(awxa)^*$, i.e., $a=awxa$. As $xS= aS$, then there is some $t\in S$ such that $x=at=(awxa)t=awx(at)=awx^2$.

Therefore, $x$ is the $w$-core inverse of $a$.
\hfill$\Box$
\end{proof}

\begin{remark} {\rm Theorem \ref{relations b,c-inverses} also works for the right hybrid $(b,c)$-inverse of $a$ \cite{Drazin2021} instead of its $(b,c)$-inverse. A list of criteria for right hybrid $(b,c)$-inverses can be found in \cite{Drazin2012,Zhu20180,Zhu2021}.}
\end{remark}

\section{Characterizations of $w$-core inverses and dual $v$-core inverses by units in a $*$-ring}

In this section, we assume that $R$ is a unital $*$-ring, and we mainly derive the existence criteria of $w$-core invertible and dual $v$-core invertible elements by units in $R$.

An element $p\in R$ is called a Hermitian element if $p^*=p$. If in addition, $p=p^2$, then $p$ is called a projection. By $R^{-1}$ we denote the group of all invertible elements in $R$. The classical generalized inverses, including the group inverse, the Moore-Penrose inverse and the core inverse are characterized as follows, respectively. (i) $a\in R^\#$ if and only if there exists an idempotent $e\in R$ such that $ea=ae=0$ and $a+e\in R^{-1}$ (see \cite[Proposition 8.24]{Rao2002}). (ii) $a\in R^\dag$ if and only if there exists a projection $p\in R$ such that $pa=0$ and $aa^*+p\in R^{-1}$ (see \cite[Theorem 1.2]{Patricio2010}). (iii) $a\in R^{\tiny\textcircled{\tiny{\#}}}$ if and only if there exists a unique projection $q\in R$ such that $qa=0$ and $a^n+q\in R^{-1}$ for any integer $n\geq 1$ (see \cite[Theorems 3.3 and 3.4]{Li2018}). Inspired by these, we aim to give the characterization and the representation of the $w$-core inverse by projections and units.

\begin{theorem}\label{idempotent} Let $a,w\in R$. The following conditions are equivalent{\rm:}

\emph{(i)} $a\in R_w^{\tiny\textcircled{\tiny{\#}}}$.

\emph{(ii)} There exists a unique projection $p\in R$ such that $pa=0$ and $u=p+aw\in R^{-1}$.

\emph{(iii)} There exists a projection $p\in R$ such that $pa=0$ and $u=p+aw\in R^{-1}$.

In this case, $a_w^{\tiny\textcircled{\tiny{\#}}}=u^{-1}(1-p)$.
\end{theorem}

\begin{proof}
(i) $\Rightarrow$ (ii) Given $a\in R_w^{\tiny\textcircled{\tiny{\#}}}$, then, by Theorem \ref{ideal form}, $aw\in R^{\tiny\textcircled{\tiny{\#}}}$ and $a\in awR$. So, there exists a unique projection $p\in R$ such that $paw=0$ and $aw+p\in R^{-1}$. Since $a\in awR$, the equality $paw=0$ implies $pa=0$.

(ii) $\Rightarrow$ (iii) is obvious.

(iii) $\Rightarrow$ (i) Given (iii), then $aw\in R^{\tiny\textcircled{\tiny{\#}}}$, i.e., $aw\in R^\#\cap R^{(1,3)}$, so that $aw\in (aw)^2R$. Note that the implication $(p+aw)a=awa \Rightarrow a=(p+aw)^{-1}awa$. Then we have $a\in (p+aw)^{-1}((aw)^2R)a=awRa \subseteq awR$. Applying Theorem \ref{ideal form} (iii) $\Rightarrow$ (i), $a\in R_w^{\tiny\textcircled{\tiny{\#}}}$.

We next show that $x=u^{-1}(1-p)=(p+aw)^{-1}(1-p)$ is the $w$-core inverse of $a$. Indeed, we have

(1) $xawa=(p+aw)^{-1}(1-p)awa=(p+aw)^{-1}awa=a$.

(2) Since $(1-p)(p+aw)=aw$ and $p+aw \in R^{-1}$, we have $1-p=aw(p+aw)^{-1}$ and $awx=aw(p+aw)^{-1}(1-p)=(1-p)^2=1-p=(awx)^*$.

(3) As $p(p+aw)=p$, then $p=p(p+aw)^{-1}$ and hence $(1-p)(p+aw)^{-1}=(p+aw)^{-1}-p$, so that $awx^2=(1-p)(p+aw)^{-1}(1-p)=((p+aw)^{-1}-p)(1-p)=(p+aw)^{-1}(1-p)=x$.
\hfill$\Box$
\end{proof}

To make an approach to our next results, we begin with several auxiliary lemmas.

\begin{lemma} \label{jacobson} For any $a,b\in R$, if $\alpha=1-ab$ is invertible, then so is $\beta=1-ba$. Moreover, $\beta^{-1}=1+b\alpha^{-1}a$.
\end{lemma}

The formula above in Lemma \ref{jacobson} is well known as Jacobson's Lemma. Two elements $\alpha=1-ab$ and $\beta=1-ba$ are said to form a Jacobson pair \cite{Lam2014}. More results on Jacobson pairs can be referred to \cite{Lam2014,Lam2013}.

Recall that an element $a\in R$ is (von Neumann) regular if there exists an $x\in R$ such that $a=axa$. Such an $x$ is called an inner inverse or a $\{1\}$-inverse of $a$, and is denoted by $a^-$. We herein remind the reader that $a\in R^{\parallel d}$ implies that $d$ is regular. As usual, by $a\{1\}$ we denote the set of all $\{1\}$-inverses of $a$.

\begin{lemma} \label{Mary inverse}{\rm \cite[Theorem 3.2]{Mary2013}} Let $d\in R$ be regular with $d^{-}\in d{\{1\}}$. The following conditions are equivalent{\rm:}

\emph{(i)} $a$ is invertible along $d$.

\emph{(ii)} $u=da+1-dd^{-} \in R^{-1}$.

\emph{(iii)} $v=ad+1-d^{-}d \in R^{-1}$.

In this case, $a^{\parallel d}=u^{-1}d=dv^{-1}$.
\end{lemma}

It should be noted that $u=da+1-dd^{-} $ and $v=ad+1-d^{-}d $ form a Jacobson pair as $u$ and $v$ can be written as $u=1+d(a-d^{-})$ and $v=1+(a-d^{-})d$, respectively.

\begin{lemma} {\rm \cite[Corollary 3.17]{Zhu2017}}\label{classical MP char} Let $a\in R$ be regular with $a^-\in a\{1\}$. Then the following conditions are equivalent{\rm:}

\emph{(i)} $a\in R^\dag$.

\emph{(ii)} $u=aa^*+1-aa^-\in R^{-1}$.

\emph{(iii)} $v=a^*a+1-a^-a\in R^{-1}$.

In this case, $a^\dag=(u^{-1}a)^*=(av^{-1})^*$.
\end{lemma}

\begin{theorem} \label{vw intersect} Let $a,w,v\in R$ and let $v\in R^{\parallel a}$ with $a^-\in a\{1\}$. Then the following conditions are equivalent{\rm:}

\emph{(i)} $a\in R_w^{\tiny\textcircled{\tiny{\#}}} \cap R_{v,\tiny\textcircled{\tiny{\#}}}$.

\emph{(ii)} $w\in R^ {\parallel a}$ and $a\in R^{\dag}$.

\emph{(iii)} $u=awavaa^*+1-aa^{-}\in R^{-1}$.

\emph{(iv)} $r=avawaa^*+1-aa^{-}\in R^{-1}$.

\emph{(v)} $s=wavaa^* a +1-a^{-}a\in R^{-1}$.

\emph{(vi)} $t=vawaa^*a+1-a^{-}a\in R^{-1}$.

In this case, $a_w^{\tiny\textcircled{\tiny{\#}}}=avaa^*as^{-1}(u^{-1}awava)^*$ and $a_{v,\tiny\textcircled{\tiny{\#}}}=(u^{-1}awava)^*awaa^*at^{-1}$.
\end{theorem}

\begin{proof}
(i) $\Leftrightarrow$ (ii) follows from Theorems \ref{relate to mary inverse} and \ref{relate to dual mary inverse}. (iii) $\Leftrightarrow$ (v) and (iv) $\Leftrightarrow$ (vi) by Lemma \ref{jacobson}. We next just show (ii) $\Leftrightarrow$ (iii) $\Leftrightarrow$ (iv).

(ii) $\Rightarrow$ (iii) Since $v\in R^{\parallel a}$, one can get $av+1-aa^-\in R^{-1}$ by Lemma \ref{Mary inverse}. Also, $w\in R^{\parallel a}$ gives $aw+1-aa^-\in R^{-1}$ and hence $awaa^-+1-aa^-\in R^{-1}$ by Lemma \ref{jacobson}. Thus, we have $(awaa^-+1-aa^-)(av+1-aa^-)=awav+1-aa^-\in R^{-1}$. Lemma \ref{classical MP char} ensures that $aa^*+1-aa^-\in R^{-1}$ as $a\in R^\dag$. Note that $awav+1-aa^-$ and $awavaa^-+1-aa^-$ are Jacobson pairs. Then $u=awavaa^*+1-aa^-=(awavaa^-+1-aa^-)(aa^*+1-aa^-)\in R^{-1}$.

(iii) $\Rightarrow$ (ii) Given (iii), i.e., $u=awavaa^*+1-aa^{-}\in R^{-1}$ for some $a^-\in a\{1\}$, then we have $ua=awavaa^*a$ and whence $a=u^{-1}awavaa^*a\in Raa^*a$. By Lemma \ref{MP inverse ideal char}, we have at once $a\in R^\dag$ and $a^\dag=(u^{-1}awava)^*$. This in turn gives $aa^*+1-aa^-\in R^{-1}$ from Lemma \ref{classical MP char}. We get that $awavaa^-+1-aa^-=u(aa^*+1-aa^-)^{-1}\in R^{-1}$, and consequently $awav+1-aa^-\in R^{-1}$. Note that the assumption $v\in R^{\parallel a}$ implies $av+1-aa^-\in R^{-1}$. Then $awaa^-+1-aa^-=(awav+1-aa^-)(av+1-aa^-)^{-1}\in R^{-1}$, whence $aw+1-aa^-\in R^{-1}$, which guarantees that $w\in R^{\parallel a}$ by Lemma \ref{Mary inverse}.

(ii) $\Leftrightarrow$ (iv) can be proved by a similar way of (ii) $\Leftrightarrow$ (iii).

We next derive the formulae of $a_w^{\tiny\textcircled{\tiny{\#}}}$ and $a_{v,\tiny\textcircled{\tiny{\#}}}$. Since $s=wavaa^* a +1-a^{-}a\in R^{-1}$, we have
$as=awavaa^*a$ and $a=awavaa^*as^{-1}$. As $w^{\parallel a}$ exists, then $w^{\parallel a}=avaa^*as^{-1}$ by Lemma \ref{Mary left and right}. So, $a_w^{\tiny\textcircled{\tiny{\#}}}=w^{\parallel a}a^{(1,3)}=w^{\parallel a}a^\dag=avaa^*as^{-1}(u^{-1}awava)^*$.

Similarly, $a_{v,\tiny\textcircled{\tiny{\#}}}=a^{(1,4)}v^{\parallel a}=(u^{-1}awava)^*awaa^*at^{-1}$. \hfill$\Box$
\end{proof}

Let us now pause to make some remarks.

\begin{remark} {\rm (1) In Theorem \ref{vw intersect} above, the assumption ``$v\in R^{\parallel a}$'' cannot be dropped. Such as, let $R$ and the involution $*$ be the same as that of Remark \ref{rmk1} above. Take $w=\sum_{i=1}^{\infty}e_{i,i+1}$, $v=\sum_{i=1}^{\infty}e_{i+1,i}$ and $a=1$, then $v \notin R^{\parallel a}$ since $v\notin R^{-1}$. By a direct check, the condition (iii) $awavaa^*+1-aa^-=wv=1\in R^{-1}$ cannot imply that $w\in R^ {\parallel a}$ (in the item (ii)) as $w\notin R^{-1}$.

(2) The equivalences among the items (i), (iii), (iv), (v) and (vi) in Theorem \ref{vw intersect} also hold without the assumption ``$v\in R^{\parallel a}$'' when $R$ is a Dedekind-finite ring satisfying the property $xy = 1 \Rightarrow yx = 1$ for any $x,y\in R$ (see Theorem \ref{vw intersect dedekind} below).

(3) The formula of $a_w^{\tiny\textcircled{\tiny{\#}}}$ can also be given by Corollary \ref{extended repre}. From $s=wavaa^* a +1-a^{-}a\in R^{-1}$, it follows that $was=(wa)^2vaa^* a$ and hence $wa=(wa)^2vaa^* as^{-1}$. As $wa\in R^\#$, then $(wa)^\#=wa(vaa^* as^{-1})^2=(wavaa^* as^{-1})vaa^* as^{-1}$. So, $a(wa)^\#=(a wavaa^* as^{-1})vaa^* as^{-1}=avaa^* as^{-1}$. As a consequence, $a_w^{\tiny\textcircled{\tiny{\#}}}=a(wa)^\#a^{(1,3)}=avaa^*as^{-1}(u^{-1}awava)^*$.}
\end{remark}

It follows from \cite{Mary2011} that $a\in R^\dag$ if and only if $a$ is invertible along $a^*$ if and only if $a^*$ is invertible along $a$ if and only if $a\in aa^*R \cap Ra^*a$. Combining with Lemma \ref{Mary inverse} and Theorem \ref{vw intersect}, the criteria for both $w$-core invertible and dual $v$-core invertible elements can be given by Mary's inverse along an element in $R$.

\begin{corollary} \label{vw} Let $a,w,v\in R$ and let $v\in R^{\parallel a}$. Then the following conditions are equivalent{\rm:}

\emph{(i)} $a\in R_w^{\tiny\textcircled{\tiny{\#}}} \cap R_{v,\tiny\textcircled{\tiny{\#}}}$.

\emph{(ii)} $w, a^*\in R^{\parallel a}$.

\emph{(iii)} $wavaa^*\in R^{\parallel a}$.

\emph{(iv)} $vawaa^*\in R^{\parallel a}$.

In this case, $a_w^{\tiny\textcircled{\tiny{\#}}}=avaa^*ax(yawava)^*$ and $a_{v,\tiny\textcircled{\tiny{\#}}}=(tavawa)^*awaa^*as$, where $x,y\in R$ satisfy $a=awavaa^*ax=yawavaa^*a$, and $s,t\in R$ satisfy $a=avawaa^*as=tavawaa^*a$.
\end{corollary}

\begin{proof} As $wavaa^*\in R^ {\parallel a}$, then $a\in awavaa^*aR \cap Rawavaa^*a$. There exist $x,y\in R$ such that $a=awavaa^*ax=yawavaa^*a$. Since $w\in R^ {\parallel a}$, we have $w^{\parallel a}=avaa^*ax$. Recall that if $a=xa^*a$ then $x^*$ is a $\{1,3\}$-inverse of $a$.
We have $a^{(1,3)}=(yawava)^*$ by $a=yawavaa^*a$. So, $a_w^{\tiny\textcircled{\tiny{\#}}}=w^{\parallel a}a^{(1,3)}=avaa^*ax(yawava)^*$.

Similarly, we have $a_{v,\tiny\textcircled{\tiny{\#}}}=(tavawa)^*awaa^*as$, where $s,t\in R$ satisfy $a=avawaa^*as=tavawaa^*a$.
\hfill$\Box$
\end{proof}

As a special case of Theorem \ref{vw intersect}, we have the following corollary.

\begin{corollary} \label{referee suggest} Let $a,w,v\in R$  with $a^-\in a\{1\}$. Then the following conditions are equivalent{\rm:}

\emph{(i)} $a\in R_w^{\tiny\textcircled{\tiny{\#}}} \cap R_{v,\tiny\textcircled{\tiny{\#}}}$.

\emph{(ii)} $w,v\in R^ {\parallel a}$ and $a\in R^{\dag}$.

\emph{(iii)} $v\in R^{\parallel a}$ and $u=awaa^*+1-aa^{-}\in R^{-1}$.

\emph{(iv)} $v\in R^{\parallel a}$ and $r=aa^*aw+1-aa^{-}\in R^{-1}$.

\emph{(v)} $v\in R^{\parallel a}$ and $s=waa^*a +1-a^{-}a\in R^{-1}$.

\emph{(vi)} $v\in R^{\parallel a}$ and $t=a^*awa+1-a^{-}a\in R^{-1}$.

In this case, $a_w^{\tiny\textcircled{\tiny{\#}}}=aa^*as^{-1}(u^{-1}awa)^*$.
\end{corollary}

\begin{theorem} \label{vw intersect dedekind} Let $R$ be a Dedekind-finite ring and let $a,w,v\in R$. Suppose $a$ is regular and $a^-\in a\{1\}$. Then the following conditions are equivalent{\rm:}

\emph{(i)} $a\in R_w^{\tiny\textcircled{\tiny{\#}}} \cap R_{v,\tiny\textcircled{\tiny{\#}}}$.

\emph{(ii)} $wav\in R^ {\parallel a}$ and $a\in R^{\dag}$.

\emph{(iii)} $vaw\in R^ {\parallel a}$ and $a\in R^{\dag}$.

\emph{(iv)} $u=awavaa^*+1-aa^{-}\in R^{-1}$.

\emph{(v)} $r=avawaa^*+1-aa^{-}\in R^{-1}$.

\emph{(vi)} $s=wavaa^* a +1-a^{-}a\in R^{-1}$.

\emph{(vii)} $t=vawaa^*a+1-a^{-}a\in R^{-1}$.

In this case, $a_w^{\tiny\textcircled{\tiny{\#}}}=avaa^*as^{-1}(u^{-1}awava)^*$ and $a_{v,\tiny\textcircled{\tiny{\#}}}=(u^{-1}awava)^*awaa^*at^{-1}$.
\end{theorem}

\begin{proof}

(i) $\Rightarrow$ (ii) Given $a\in R_w^{\tiny\textcircled{\tiny{\#}}} \cap R_{v,\tiny\textcircled{\tiny{\#}}}$, it follows from Proposition \ref{wv core char} that $w, v\in R^ {\parallel a}$ and $a\in R^{\dag}$. In terms of Lemma \ref{Mary left and right}, we know that $w\in R^ {\parallel a}$ implies that $a\in awaR \cap Rawa$, and that $v\in R^ {\parallel a}$ gives $a\in avaR \cap Rava$. Hence, $a\in awavaR \cap Rawava$, i.e., $wav\in R^{\parallel a}$.

(ii) $\Rightarrow$ (i) As $wav\in R^ {\parallel a}$, then $a\in awavaR \cap Rawava$ by Lemma \ref{Mary left and right}. In addition, $a\in awavaR\subseteq awaR$ shows that $w$ is right invertible along $a$ (see \cite[Theorem 2.4]{Zhu2016}). Note the fact that $w$ is right invertible along $a$ implies that $w$ is left invertible along $a$ in a Dedekind-finite ring $R$. Then $w\in R^{\parallel a}$. Similarly, $v\in R^{\parallel a}$.

(i) $\Leftrightarrow$ (iii) can be proved similarly.

(ii) $\Leftrightarrow$ (iv) and (iii) $\Leftrightarrow$ (v) are similar to the proof of Theorem \ref{vw intersect}.

(iv) $\Leftrightarrow$ (vi) and (v) $\Leftrightarrow$ (vii) follow from Lemma \ref{jacobson}. \hfill$\Box$
\end{proof}

As a consequence of Proposition \ref{wv core char} and Theorem \ref{vw intersect dedekind}, the characterization about the product along an element can be obtained.

\begin{corollary} Let $R$ be a Dedekind-finite ring and let $a,w,v\in R$. Then the following conditions are equivalent{\rm:}

\emph{(i)} $w,v\in R^{\parallel a}$.

\emph{(ii)} $wav\in R^{\parallel a}$.
\end{corollary}

For any $a,w,v\in R$ with $a$ regular, taking $v=a^*$ in Corollary \ref{referee suggest}, we get $a\in R_w^{\tiny\textcircled{\tiny{\#}}} \cap R_{a^*,\tiny\textcircled{\tiny{\#}}}$ if and only if $w\in R^ {\parallel a}$ and $a\in R^{\dag}$ if and only if $a\in R_w^{\tiny\textcircled{\tiny{\#}}} \cap R_{w,\tiny\textcircled{\tiny{\#}}}$.

The following result, presents the characterization of the $w$-core inverse and the dual $w$-core inverse by units, whose proof is left to the reader.

\begin{theorem} \label{intersect} Let $a,w\in R$ and let $a$ be regular with $a^-\in a\{1\}$. Then the following conditions are equivalent{\rm:}

\emph{(i)} $a\in R_w^{\tiny\textcircled{\tiny{\#}}} \cap R_{w,\tiny\textcircled{\tiny{\#}}}$.

\emph{(ii)} $w\in R^{\parallel a}$ and $a\in R^{\dag}$.

\emph{(iii)} $a\in R_w^{\tiny\textcircled{\tiny{\#}}} \cap R_{a^*,\tiny\textcircled{\tiny{\#}}}$

\emph{(iv)} $u=awaa^* +1-aa^{-}\in R^{-1}$.

\emph{(v)} $r=a^*awa +1-a^{-}a\in R^{-1}$.

\emph{(vi)} $s=waa^*a +1-a^{-}a\in R^{-1}$.

\emph{(vii)} $t=aa^*aw +1-aa^{-}\in R^{-1}$.

In this case, $a_w^{\tiny\textcircled{\tiny{\#}}}=t^{-1}aa^*$ and $a_{w,\tiny\textcircled{\tiny{\#}}}=a^* as^{-1}$.
\end{theorem}

\begin{corollary} {\rm \cite[Theorem 5.6]{Chen2017}} Let $a\in R$ be regular with $a^-\in a\{1\}$. Then the following conditions are equivalent{\rm:}

\emph{(i)} $a\in R^{\tiny\textcircled{\tiny{\#}}} \cap R_{\tiny\textcircled{\tiny{\#}}}$.

\emph{(ii)} $a\in R^\# \cap R^{\dag}$.

\emph{(iii)} $u=a^2a^* +1-aa^{-}\in R^{-1}$.

\emph{(iv)} $v=a^* a^2 +1-a^{-}a\in R^{-1}$.

\emph{(v)} $s=aa^* a +1-a^{-}a\in R^{-1}$.

\emph{(vi)} $t=aa^* a +1-aa^{-}\in R^{-1}$.

In this case, $a^{\tiny\textcircled{\tiny{\#}}}=t^{-1}aa^*$ and $a_{\tiny\textcircled{\tiny{\#}}}=a^* as^{-1}$.
\end{corollary}

\section{Applications to complex matrices}

In this section, we particularize $S$ to $M_n(\mathbb{C})$, the ring of all $n \times n$ complex matrices, then we can obtain the concise criteria for the $W$-core inverse of $A$, where $A,W\in M_n(\mathbb{C})$. It should be mentioned that the notion of the $W$-core inverse of $A$ was not considered so far in the context of complex matrices.

\begin{theorem} Let $A,W,X\in M_n(\mathbb{C})$. Then the following conditions are equivalent{\rm:}

\emph{(i)} $X$ is the $W$-core inverse of $A$.

\emph{(ii)} $AWX=P_A$ and $\mathcal{R}(X) = \mathcal{R}(A)$.

\emph{(iii)} $AWX=P_A$ and $\mathcal{R}(X) \subseteq \mathcal{R}(A)$.
\end{theorem}

\begin{proof} (i) $\Rightarrow$ (ii) As $X$ is the $W$-core inverse of $A$, then, by the implication (i) $\Rightarrow$ (iii) of Theorem \ref{characteristic ew}, we have $XM_n=AM_n$. So, $\mathcal{R}(X) = \mathcal{R}(A)$. It follows from Theorem \ref{relate to mary inverse} that $X=W^{
\parallel A}A^\dag$ and hence $AWX=AWW^{
\parallel A}A^\dag=AA^\dag=P_A$.

(ii) $\Rightarrow$ (iii) is a tautology.

(iii) $\Rightarrow$ (i) Given $AWX=P_A$, then $AWX=(AWX)^*$. There exists some $Y\in M_n(\mathbb{C})$ such that $X=AY$ since $\mathcal{R}(X) \subseteq \mathcal{R}(A)$. Also, it follows from $AWX=P_A$ that $AWXA=P_AA=A$ and and hence $AWX^2=AWXAY=AY=X$. Note that $A^*=(AWXA)^*=A^*(AWX)^*=A^*AWX$ and $\mathcal{R}(X) \subseteq \mathcal{R}(A)$ (i.e., $\mathcal{C}(A^*)$, where $\mathcal{C}(A^*)$ denotes the row space of $A^*$). Then there is some $Z\in M_n(\mathbb{C})$ such that $X=ZA^*=ZA^*AWX=XAWX$.  Observe that ${\rm rank}(A)={\rm rank}(AWXA)\leq {\rm rank}(X)$. Then $\mathcal{R}(X) = \mathcal{R}(A)$ and $A=XN$ for some $N\in M_n(\mathbb{C})$. We hence have $A=XN=XAWXN=XAWA$, as required.
\hfill$\Box$
\end{proof}

As a consequence of Theorems \ref{relate to mary inverse} and \ref{wv mary}, we have the following result.

\begin{corollary} \label{complex W} Let $A,W\in M_n(\mathbb{C})$. Then

\emph{(i)} $A$ is $W$-core invertible if and only if $W$ is invertible along $A$ if and only if $AW$ is invertible along $AA^*$. In this case, $A_{W}^{\tiny\textcircled{\tiny{\#}}}=W^{\parallel A}A^\dag=(AW)^{\parallel AA^*}$.

\emph{(ii)} If $W^{\parallel A}$ exists and $\lambda \in \mathbb{C}/ \{0\}$, then $(\lambda A)_{W}^{\tiny\textcircled{\tiny{\#}}}=\frac{1}{\lambda}A_{W}^{\tiny\textcircled{\tiny{\#}}}$.
\end{corollary}

In terms of Lemma \ref{Mary left and right} and Corollary \ref{complex W}, the existence criterion for the $W$-core inverse of $A$ is given below by the rank in complex matrices.

\begin{theorem} Let $A,W\in M_n(\mathbb{C})$. Then $A$ is $W$-core invertible if and only if ${\rm rank} (A)={\rm rank} (AWA)$. In this case, $A_{W}^{\tiny\textcircled{\tiny{\#}}}=A(AWA)^\dag AA^\dag$.
\end{theorem}

\bigskip
\centerline {\bf ACKNOWLEDGMENTS}
\vskip 2mm
The authors are grateful to the referee for his/her careful reading and valuable comments which led to the improvement in this paper. It was notably suggested by the referee to divide the article into three parts, each one discussing a different algebraic structure. Some other results, such as Proposition \ref{core another 1} (ii) $\Leftrightarrow$ (iii) and the formulae, Remark \ref{referee add} were given by the referee. Also, the proof of Theorem \ref{idempotent} is simplified by the referee. This research is supported by the National Natural Science Foundation of China (No. 11801124, No. 11771076) and China Postdoctoral Science Foundation (No. 2020M671068).

\begin{flushleft}
{\bf References}
\end{flushleft}

\end{document}